\documentclass{amsart}

\usepackage{amsmath}
\usepackage{amsxtra}
\usepackage{amscd}
\usepackage{amsthm}
\usepackage{amsfonts}
\usepackage{amssymb}
\usepackage{mathtools}
\usepackage{eucal}
\usepackage[all]{xy}
\usepackage{graphicx}
\usepackage{tikz-cd}
\usepackage{mathrsfs}
\usepackage{euler}
\usepackage{color}
\usepackage{caption}
\usepackage{enumerate}
\usepackage{adjustbox}
\usepackage{faktor}
\usepackage{bbm}
\usepackage{stmaryrd}
\usepackage[T1]{fontenc} 
\usepackage[utf8]{inputenc} 
\usepackage{tikz}
\usepackage{tikz-cd}
\usetikzlibrary{matrix,arrows,decorations.pathmorphing}
\usepackage{hyperref}
\usepackage{float}

\usepackage[b5paper, bottom=3cm, top=3cm, right=2.5cm, left=2.5cm]{geometry}

\theoremstyle{plain}
  \newtheorem{theorem}{Theorem}
  \newtheorem{thm}[theorem]{Theorem}

  \newtheorem{lem}[theorem]{Lemma}

\theoremstyle{definition}

  \newtheorem{conjecture}[theorem]{Conjecture}
  \newtheorem{conj}[theorem]{Conjecture}
  
\theoremstyle{remark}
  \newtheorem{rem}[theorem]{Remark}

\DeclareMathAlphabet{\mathcal}{OMS}{cmsy}{m}{n}
\DeclareFontShape{OT1}{cmr}{m}{up}
      {<->ssub*cmr/m/n}{}

\newcommand{\C}{\mathbb{C}}

\newcommand{\G}{\mathbb{G}}

\newcommand{\N}{\mathbb{N}}

\newcommand{\Q}{\mathbb{Q}}
\newcommand{\R}{\mathbb{R}}

\newcommand{\Z}{\mathbb{Z}}

\newcommand{\calo}{\mathcal{O}}

\newcommand{\calu}{\mathcal{U}}

\newcommand{\mfa}{\mathfrak{a}}

\newcommand{\mfp}{\mathfrak{p}}

\newcommand{\mfs}{\mathfrak{S}}
\newcommand{\rmn}{\mathrm{N}}

\newcommand{\rmt}{\mathrm{T}}

\newcommand{\ab}{{\mathrm{ab}}}
\newcommand{\Ad}{{\mathrm{Ad}}}

\newcommand{\spec}{\mathrm{Spec}\,}

\newcommand{\Id}{{\mathrm{Id}}}

\renewcommand{\Im}{{\mathrm{Im}}}

\newcommand{\gal}{\mathrm{Gal}}
\newcommand{\Gal}{{\mathrm{Gal}}}

\newcommand{\Hom}{{\mathrm{Hom}}}

\newcommand{\Ind}{{\mathrm{Ind}}}

\newcommand{\Reg}{{\mathrm{Reg}}}

\newcommand{\GL}{{\mathrm{GL}}}

\newcommand{\et}{{\mathrm{\acute{e}t}}}
\newcommand{\fl}{{\mathrm{fl}}}
\newcommand{\Zar}{{\mathrm{Zar}}}

\newcommand{\Frob}{{\mathrm{Frob}}}

\newcommand{\ord}{{\mathrm{ord}}}
\newcommand{\Tr}{{\mathrm{Tr}}}

\newcommand{\cusp}{\mathrm{Cusp}}

\newcommand{\rs}{\mathrm{RS}}

\newcommand{\stark}{{\mathrm{Stark}}}
\newcommand{\Stark}{{\mathrm{Stark}}}
\newcommand{\SD}{{\mathrm{SD}}}

\newcommand{\lra}{{\, \longrightarrow \,}}
\newcommand{\iso}{\, \xrightarrow{\widesim{}} \,}

\newcommand{\paren}[1]{\mathopen{}\left(#1\right)\mathclose{}}
\newcommand{\set}[1]{\mathopen{}\left\{#1\right\}\mathclose{}}
\newcommand{\sbrac}[1]{\mathopen{}\left[#1\right]\mathclose{}}
\newcommand{\abrac}[1]{\mathopen{}\left\langle#1\right\rangle\mathclose{}}
\newcommand{\verts}[1]{\mathopen{}\left\lvert#1\right\rvert\mathclose{}}
\newcommand{\norm}[1]{\mathopen{}\left\lvert\left\lvert#1\right\rvert\right\rvert\mathclose{}}
\newcommand\restr[2]{{
  \left.\kern-\nulldelimiterspace 
  #1 
  \right|_{#2} 
  }}

\newcommand{\Mid}{\,\middle|\,}

\newcommand{\pair}[1]{\abrac{#1}}
\newcommand{\wh}{\widehat}

\newcommand{\widesim}[2][2]{
  \mathrel{\overset{#2}{\scalebox{#1}[1]{$\sim$}}}
}

\renewcommand{\setminus}{\backslash}

\makeatletter
\AtBeginDocument
 {
    \def\@thm#1#2#3{%
      \ifhmode
        \unskip\unskip\par
      \fi
      \normalfont
      \trivlist
      \let\thmheadnl\relax
      \let\thm@swap\@gobble
      \let\thm@indent\indent 
      \thm@headfont{\scshape}
      \thm@notefont{\fontseries\mddefault\upshape}%
      \thm@headpunct{.}
      \thm@headsep 5\p@ plus\p@ minus\p@\relax
      \thm@space@setup
      #1
      \@topsep \thm@preskip               
      \@topsepadd \thm@postskip           
      \def\dth@counter{#2}%
      \ifx\@empty\dth@counter
        \def\@tempa{%
          \@oparg{\@begintheorem{#3}{}}[]%
        }%
      \else
        \H@refstepcounter{#2}%
        \hyper@makecurrent{#2}%
        \let\Hy@dth@currentHref\@currentHref
        \AddToHookNext{para/begin}{\MakeLinkTarget*{\Hy@dth@currentHref}}%
        \def\@tempa{%
          \@oparg{\@begintheorem{#3}{\csname the#2\endcsname}}[]%
        }%
      \fi
      \@tempa
    }%
  \dth@everypar={%
    \@minipagefalse \global\@newlistfalse
    \@noparitemfalse
    \if@inlabel
      \global\@inlabelfalse
      \begingroup \setbox\z@\lastbox
       \ifvoid\z@ \kern-\itemindent \fi
      \endgroup
      \unhbox\@labels
    \fi
    \if@nobreak \@nobreakfalse \clubpenalty\@M
    \else \clubpenalty\@clubpenalty \everypar{}%
    \fi
  }%
}

\title{Towards the $p$-adic derived Hecke algebra for weight one forms}

\author{Robin Zhang}
\address{Department of Mathematics, Massachusetts Institute of Technology}
\email{robinz@mit.edu}

\date{June 10, 2025}

\begin{document}

\begin{abstract}
  This note outlines an approach to defining $p$-adic Shimura classes and $p$-adic derived Hecke operators on the completed cohomology of modular curves from upcoming work by the author. After reviewing the modulo-$p$ constructions of Harris and Venkatesh, we formulate a conjecture relating the action of $p$-adic derived Hecke operators on cusp forms of weight $1$ and level $\Gamma_1(N)$ to the $p$-adic logarithm of the Stark unit for the corresponding adjoint Deligne--Serre representation. This new $p$-adic conjecture can be viewed as complementary to the Harris--Venkatesh conjecture.
\end{abstract}

\maketitle

\setcounter{tocdepth}{1}
\tableofcontents


\sloppy
\section{Introduction}

Let $f = \sum_n a_n q^n$ be a newform of weight $1$ and level $\Gamma_1(N)$,
and denote its field of coefficients by $\Q(f) \subset \C$
with ring of integers $\Z[f]$.
There is a finite Galois extension $E/\Q$
and an associated $3$-dimensional complex representation
$\Ad(\rho_f)$ of $\gal(E/\Q)$ obtained by taking the adjoint action
of the Deligne--Serre representation $\rho_f$ of $f$
on trace-free $2 \times 2$ complex matrices.
The trace-free adjoint representation has an associated dual space of units
$\calu(\Ad(\rho_f)) :=
\Hom_{\Gal(E/\Q)}(\Ad(\rho_f), \calo_E^\times \otimes \Z[f])$
with a $(\Z/p\Z)^\times$-regulator map
\[
  \Reg_{(\Z/p\Z)^\times} : \calu\big(\Ad(\rho_f)\big) \lra (\Z/p\Z)^\times \otimes \Z\sbrac{f, \frac{1}{6N}}
\]
for each prime $p \nmid 6N$.
Let $\cusp$ denote the cuspidal divisor
on the modular curve $X_1(N)$.
Using constructions (especially the Shimura covering/subgroup/class)
with origins in the works
of Shimura \cite{shimura-1963,shimura-1967}, Mazur \cite{mazur},
and Merel \cite{merel2} among others,
the study of derived Hecke operators on
the coherent cohomology of modular curves
\[
  T_{(\Z/p\Z)^\times, N}: H^0\big(X_1(N), \omega(-\cusp)\big) \otimes (\Z/p\Z)^\times
    \lra
    H^1\big(X_1(N), \omega(-\cusp)\big) \otimes (\Z/p\Z)^\times
\]
was initiated by Harris and Venkatesh \cite{hv},
who posed a conjecture relating the action of 
$T_{(\Z/p\Z)^\times, N}$
to the action of a group of units.
For the newform $f^* = \sum_n \overline{a_n} q^n$ obtained from $f$ by
complex conjugation of Fourier coefficients,
the Serre-duality pairing
$\langle f^*, T_{(\Z/p\Z)^\times, N}(f)\rangle_{\SD}$
(i.e. ``pseudo-eigenvalue'' of $T_{(\Z/p\Z)^\times, N}$)
was packaged as the Harris--Venkatesh ``norm''
$\norm{f}_{(\Z/p\Z)^\times}^2 \in (\Z/p\Z)^\times \otimes \Z[f]$
in \cite{zhang-hvs},
which drew an analogy with the Petersson norm
$\norm{f}_{\R}^2 = \int_{X_0(N)} \verts{f}^2 \frac{dxdy}{y} \in \R$
and the conjectures of Stark \cite{stark-1971,stark-1975,stark-1976,stark-1980}.
\begin{conjecture}[The Harris--Venkatesh conjecture]
  \label{conj:hv}
  Let $f$ be a newform of weight $1$ and level $\Gamma_1(N)$.
  There exists a $u \in \calu(\Ad(\rho_f)) \otimes \Q$
  and a prime $p_0$ such that for all primes $p \geq p_0$,
  \[
    \norm{f}_{(\Z/p\Z)^\times}^2 = \Reg_{(\Z/p\Z)^\times}(u)
  \]
  or equivalently,
  \[
    \abrac{f^*, T_{\Z_p, N}(f)}_{\SD} = \Reg_{(\Z/p\Z)^\times}(u).
  \]
\end{conjecture}

The Harris--Venkatesh conjecture has been proved
for imaginary dihedral forms in \cite{dhrv, lecouturier-hv, zhang-hv},
and for certain real dihedral forms in \cite{dhrv}.
Numerical evidence for some exotic $f$ has been given
in \cite{marcil}.
The Harris--Venkatesh conjecture
has been refined in \cite{zhang-hvs} to be compatible
with the Stark conjecture (i.e. $u$ can be taken to be the Stark unit $u_\Stark$)
and in \cite{zhang-rs} with explicit local data
(i.e. the minimal $p_0$).

This note outlines key constructions
and a conjecture from upcoming work of the author
on $p$-adic derived Hecke operators.
Section \ref{sec:stark-unit-group} briefly reviews
the definition of the Stark unit group
and the regulator maps defined on it.
Section \ref{sec:mod-p} reviews
the modulo-$p$ derived Hecke operator
and the Harris--Venkatesh conjecture.
Section \ref{sec:p-adic-forms}
gives the construction of an element $\wh{f}^*$
of the completed cohomology group
\[
  \wh{H}^0 \big(X_0(p^\infty), \omega(-\cusp)\big)
    = \varinjlim_n H^0 \big(X_0(p^n), \omega(-\cusp)\big),
\]
which gives rise to an element
$f \cdot \wh{f}^* \in \wh{H}^0(X_0(p^\infty), \Omega)$
that will be paired with the $p$-adic Shimura class
and serve as an analogue of the weight-$2$ cusp form
$\Tr_{\Gamma_0(p)}^{\Gamma_0(p) \cap \Gamma_1(N)}(f(z) f^*(pz))$
in \cite[Section 1.2]{dhrv}.
Section \ref{sec:p-adic-Shimura} gives the construction of
$p$-adic Shimura classes 
\begin{align*}
  \mfs_{\Z_p^\times}
    &\in \Hom_{\Z_p^\times}(\wh{H}^0 (X_0(p^\infty), \Omega), \Z_p^\times) \\
  \mfs_{\Z_p}
    &\in \Hom_{\Z_p}(\wh{H}^0 (X_0(p^\infty), \Omega), \Z_p)
\end{align*}
via projective limits of flat cohomology classes
obtained from modular curve coverings
\[
  \pi_n: X_1 \paren{p^n} \lra X_0\paren{p^n}.
\]
Since $\mfs_{\Z_p^\times} = \mfs_{(\Z/p\Z)^\times} \times \mfs_{\Z_p}$,
we choose to focus on the $\Z_p$-component
that is complementary to the Harris--Venkatesh setting.
Serre duality is applied to $f \cdot \wh{f}^*$
and the Shimura class $\mfs_{\Z_p}$
to define the $\Z_p$ derived Hecke operator
\[
  T_{\Z_p, N}: H^0\big(X_1(N), \omega(-\cusp)\big)^\ord
    \lra H^1\big(X_1(N), \omega(-\cusp)\big)^\ord,
\]
assuming ordinariness when $p = 2$ or $3$,
i.e. $\ord_v (a_p) = 0$ for a good prime $v$ over $p$;
this gives a a $p$-adic ``norm'' 
$\norm{f}_{\Z_p}^2 \in \Z_p \otimes \calo_v$
where $\calo_v$ is the completion of $\Z[f]$ at a finite prime $v$ over $p$.
Inspired by the Harris--Venkatesh conjecture,
we propose the following
$p$-adic conjecture
in terms of the
$p$-adic regulator $\Reg_{\Z_p}$,
which is the $p$-adic logarithm of a unit
after evaluation at a distinguished element
$w_v \in \Ad(\rho_f)$
(cf. the relationship between $\Reg_\R$
and the classical logarithm in \cite{tate-stark,zhang-hvs}).
\begin{conj}[$p$-adic conjecture]
  \label{conj:p-adic-hv}
  Let $f$ be a newform of weight $1$ and level $\Gamma_1(N)$.
  There exists an element $u \in \calu(\Ad(\rho_f)) \otimes \Q$
  and a prime $p_0$ such that for all primes $p \geq p_0$,
  \[
    \norm{f}_{\Z_p}^2 = \Reg_{\Z_p}(u),
  \]
  or equivalently,
  \[
    \abrac{f^*, T_{\Z_p, N}(f)} = \Reg_{\Z_p}(u).
  \]
\end{conj}

In the sense that the original conjecture of Harris--Venkatesh \cite{hv}
is over $(\Z/p\Z)^\times$ and that there are isomorphisms
$\Z_p^\times \cong \mu_{p-1} \times (1 + p\Z_p)
\cong (\Z/p\Z)^\times \times \Z_p$,
Conjecture \ref{conj:p-adic-hv} is a complement to the
Harris--Venkatesh conjecture.

\begin{rem}
  There is a unique element $u_f \in \calu(\Ad(\rho_f)) \otimes \Q$
  associated to $f$ by \cite[Conjecture 1]{zhang-hvs}.
  This $u_f$ is a multiple of the Stark unit $u_\stark$ from the Stark conjecture
  for $\Ad(\rho_f)$ by a Rankin--Selberg constant
  $c_{f, \rs}$ (cf. \cite[Theorem 3]{zhang-hvs}, \cite{zhang-rs}).
  It would be nice to specify how the $u$
  in the $p$-adic conjecture \ref{conj:p-adic-hv}
  is related to $u_f$,
  as was done for the mod-$p$ setting in
  \cite{zhang-hvs}.
  The author would like to speculate that
  perhaps in this way,
  one could draw a precise local--global picture
  for Stark units via the derived Hecke algebra.
  Furthermore, the author hopes to relate the new
  $p$-adic conjecture
  to special values of $p$-adic $L$-functions;
  this would be an analogue to the Gross--Stark conjecture
  with units instead of $p$-units,
  perhaps answering the speculations of
  Rivero \cite[Sections 1.3 and 6]{rivero}.
\end{rem}

\subsection*{Acknowledgements}

The author is grateful to Kenichi Namikawa
and Keiichi Gunji
for the opportunity to present this work at the
``Arithmetic aspects of automorphic forms and automorphic representations''
conference at the
Research Institute for Mathematical Sciences, Kyoto University (RIMS)
and for the inclusion of this report in its proceedings.
The author thanks
Michael Harris, Lo\"{i}c Merel,
Gyujin Oh, and Akshay Venkatesh for
helpful discussions on this topic.

This material is based on work
supported by the National Science Foundation
under Grant No. DMS-2303280.


\section{The Stark unit group}
\label{sec:stark-unit-group}

The ``Stark unit group'' defined in this section
follows the treatment and notation of
\cite{zhang-hvs}, which itself is adapted
from \cite{tate-stark,hv,dhrv}.
The elements of this group
appear in the Stark conjecture,
Conjecture \ref{conj:hv}, and
Conjecture \ref{conj:p-adic-hv}.

For a modular form $f$ of weight $1$ and
level $\Gamma_1(N)$,
there is a finite Galois extension $E/\Q$
and an associated Artin representation
$\rho_f: \Gal(E/\Q) \lra \GL(V)$
realized on a free module $V$ of rank $2$ over $\Z[f]$
by Deligne--Serre \cite{deligne-serre}.
From $\rho_f$, one obtains a $3$-dimensional complex representation
$\Ad(\rho_f)$ from the action of $\Gal(E/\Q)$
by conjugation through $\rho_f$ on trace-free $2 \times 2$ complex matrices.
For $\sigma \in \Gal(E/\Q)$,
let $x_\sigma := 2 \rho_f\paren{\sigma} -
\Tr\paren{\rho_f\paren{\sigma}} \cdot \Id_{2 \times 2}$ denote
the trace-zero $2 \times 2$ matrix obtained from $\rho_f(\sigma)$.

The (trace-free) adjoint representation $\Ad(\rho_f)$
has an integral model given by
$M := \Z[f] \cdot \set{x_\sigma \Mid \sigma \in \Gal(E/\Q)}$.
There is a (dual) space of units
\[
  \calu\big(\Ad(\rho_f)\big) := \Hom_{\Gal(E/\Q)}
    \paren{M, \calo_E^\times \otimes \Z[f]},
\]
which is called the ``Stark unit group'' in \cite{hv}
because the Stark conjectures \cite{stark-1971,stark-1975,stark-1976,stark-1980}
predict the existence of a Stark unit
$u_\Stark \in \calu(\Ad(\rho_f)) \otimes \Q$
that is unique up to multiplication by roots of unity
and whose Stark regulator $\Reg_\R(u_\Stark)$
gives the leading term of the Artin $L$-function $L(\Ad(\rho_f), s)$
at $s = 0$.
As detailed in \cite[Lemma 2.1]{hv} and \cite[Corollary 2.6]{horawa},
the space $\calu(\Ad(\rho_f))$ is a rank-$1$ $\Z[f, \frac{1}{6N}]$-module
that does not depend on the choice of $E$.

For each place $w$ of $E$, there is
\begin{itemize}
  \item an embedding $\iota_w: E \hookrightarrow E_w$
    where $E_w$ is $\R$, $\C$, or a $p$-adic field;
  \item a Frobenius element $\Frob_w$ given by the $p$-power
    map or complex conjugation; and
  \item a distinguished element
    $x_{\Frob_w} \in M$
    that is invariant under the action of the decomposition group
    $\abrac{\Frob_w} \subset \Gal(E/\Q)$.
\end{itemize}
The Stark regulator map $\Reg_\R$ can be defined by
\begin{enumerate}
  \item picking an archimedean place $w$ of $E$,
  \item taking the composition of the evaluation-at-$x_{\Frob_w}$ map,
    $\iota_w$,
    and the logarithm map:
    \[
      \begin{tikzcd}
        \Reg_{\R}:
          \calu\big(\Ad(\rho_f)\big)
            \arrow[r, "\substack{\text{evaluation} \\ \text{at } x_{\Frob_w}}"] &
          \paren{\calo_E^\times}^{\Frob_w} \otimes \Z[f]
            \arrow[r, "\log \circ \iota_w"] &
          \R \otimes \Z[f].
      \end{tikzcd}
    \]
\end{enumerate}

\noindent
For each prime $p$,
there is a $(\Z/p\Z)^\times$ regulator map
(called ``reduction of a Stark unit'' in \cite{hv, dhrv})
defined by
\begin{enumerate}
  \item picking a non-archimedean place $w$ of $E$ above $p$,
  \item taking the composition of the evaluation-at-$x_{\Frob_w}$ map,
    $\iota_w$,
    and reduction modulo the ideal of $w$:
    \[
      \begin{tikzcd}
        \Reg_{(\Z/p\Z)^\times}:
          \calu\big(\Ad(\rho_f)\big)
            \arrow[r, "\substack{\text{evaluation} \\ \text{at } x_{\Frob_w}}"] &
          \paren{\calo_E^\times}^{\Frob_w} \otimes \Z[f]
            \arrow[r, "\iota_w"] &
          \calo_{E_w}^\times \otimes \Z[f]
            \arrow[r, "\substack{\text{reduction} \\ \text{mod } w}"] &
          (\Z/p\Z)^\times \otimes \Z[f].
      \end{tikzcd}
    \]
\end{enumerate}
Harris and Venkatesh \cite[Section 2.8]{hv} observed that
$\Reg_{(\Z/p\Z)^\times}$
is independent of the choice of $w$ over $p$:
the vector $x_{\Frob_w} \in M$
is invariant under the decomposition group $\abrac{\Frob_w}$
and the maps used in defining $\Reg_{(\Z/p\Z)^\times}$
are $\Gal(E/\Q)$-equivariant, so changing $w$
corresponds to a Galois conjugation
which leaves the regulator map invariant.

By omitting the ``reduction modulo $w$'' step
in the construction of $\Reg_{{\Z/p\Z}^\times}$,
we can define two $p$-adic regulators
(with an additional application of the $p$-adic logarithm function):
$\Reg_{\Z_p}$ and $\Reg_{\Z_p^\times}$.


\section{The modulo-\texorpdfstring{$p$}{p} Shimura class and derived Hecke operator}
\label{sec:mod-p}

This section is a condensed overview of the remaining ingredients
(Shimura class, norm, and derived Hecke operator)
necessary to state Conjecture \ref{conj:hv} following
the presentation in \cite{zhang-hv,zhang-hvs}.

\subsection{The modulo-\texorpdfstring{$p$}{p} Shimura class}
Let $p$ be a prime not dividing $6N$.
The Shimura covering
is a finite \'{e}tale Galois covering of modular curves
(more precisely, the maximal \'{e}tale intermediate extension)
with deck group $(\Z/p\Z)^\times$:
\[
  \begin{tikzcd}[row sep = large]
    X_1(p) \arrow[d, "(\Z/p\Z)^\times"] \\
    X_0(p)
  \end{tikzcd}
\]
which induces an element
$\mfs_{(\Z/p\Z)^\times, \et} \in H^1_\et\paren{X_0(p), (\Z/p\Z)^\times}$.
After taking
\begin{enumerate}[(a)]
  \item the base change
  $X_0(p)_{\Z/(p-1)\Z} := X_0(p) \otimes \spec(\Z/(p-1)\Z)$,
  \item the pushforward of the \'{e}tale sheaf
  $\Z/(p-1)\Z \lra \calo_{X_0(p)_{\Z/(p-1)\Z}}$,
  \item the comparison of Zariski and \'{e}tale cohomology
  for quasi-coherent sheaves,
  and
  \item the pairing given by Serre duality,
    \[
      \begin{tikzcd}[row sep = large]
        \paren{H^1\paren{X_0(p)_{\Z/(p-1)\Z}, \calo} \otimes (\Z/p\Z)^\times}
        \otimes
        \paren{H^0\paren{X_0(p)_{\Z/(p-1)\Z}, \Omega^1} \otimes (\Z/p\Z)^\times}
          \arrow[d, "\abrac{\,\cdot\,, \,\cdot\,}_{\SD}"] \\
        (\Z/p\Z)^\times,
      \end{tikzcd}
    \]
\end{enumerate}
we have a composition of the respective maps:
\[
  \begin{tikzcd}[column sep = tiny]
    \mfs_{(\Z/p\Z)^\times, \et} \arrow[ddd, mapsto] & \in &
      H^1_\et\paren{X_0(p), (\Z/p\Z)^\times} \arrow[d, "(a)"] \\
    & &
      H^1_\et\paren{X_0(p), \Z/(p-1)\Z} \otimes_\Z (\Z/p\Z)^\times \arrow[d, "(b)"] \\
    & &
      H^1_\et\paren{X_0(p)_{\Z/(p-1)\Z}, \G_a} \otimes (\Z/p\Z)^\times
        \arrow[d, "(c)"] \\
    \mfs_{(\Z/p\Z)^\times} \arrow[d, mapsto] & \in &
      H^1_\Zar\paren{X_0(p)_{\Z/(p-1)\Z}, \calo} \otimes (\Z/p\Z)^\times
        \arrow[d, "(d)"] \\
    \abrac{\, \cdot \,, \mfs_{(\Z/p\Z)^\times}}_\SD & \in &
      \Hom\paren{H^0\paren{X_0(p), \Omega^1}, (\Z/p\Z)^\times}.
  \end{tikzcd}
\]
Hence, the image of $\mfs_{(\Z/p\Z)^\times, \et}$
under the first four maps furnishes an element
$\mfs_{(\Z/p\Z)^\times} \in
H^1\paren{X_0(p)_{\Z/(p-1)\Z}, \calo} \otimes (\Z/p\Z)^\times$.
Since there is an isomorphism
$S_2(\Gamma_0(p)) \cong H^0\paren{X_0(p), \Omega^1}$,
(see \cite[Corollary 2.17]{shimura}),
we can view $\mfs_{(\Z/p\Z)^\times}$ as acting on
weight-$2$ cusp forms via
$\langle \, \cdot \,, \mfs_{(\Z/p\Z)^\times}\rangle_\SD \in
\Hom\paren{S_2(\Gamma_0(p)), (\Z/p\Z)^\times}$.

\subsection{The modulo-\texorpdfstring{$p$}{p} norm}

For a cusp form $f = \sum_n a_n q^n$ of weight $1$ and level $\Gamma_1(N)$
with coefficients in $\Q(f) \subset \C$,
applying complex conjugation to the Fourier coefficients
gives the dual cusp form $f^* = \sum_n \overline{a_n} q^n$ of weight $1$
and level $\Gamma_1(N)$.
Define $\Gamma_{0, 1}(p, N) := \Gamma_0(p) \cap \Gamma_1(N)$
and denote its modular curve by $X_{0, 1}(p, N) = X_{\Gamma_{0, 1}(p, N)}$.

For a prime $p$ not dividing $6N$,
the product $f(z)f^*(pz)$ is a cusp form of weight $2$ and level
$\Gamma_0(p) \cap \Gamma_1(N)$ whose trace
$\Tr_{\Gamma_0(p)}^{\Gamma_{0, 1}(p, N)} \paren{f(z)f^*(pz)}$
is a cusp form of weight $2$, level $\Gamma_0(p)$,
trivial nebentypus, and with coefficients in $\Z[f, \frac{1}{6N}]$.
\cite{zhang-hvs} defines a Harris--Venkatesh ``norm''
\[
  \norm{f}^2_{(\Z/p\Z)^\times} := \mfs_p \paren{\Tr_p^{Np}\paren{f(z)f^*(pz)}}
    \in (\Z/p\Z)^\times \otimes \Z\sbrac{f, \frac{1}{6N}}.
\]
This leads to the norm
formulation of Conjecture \ref{conj:hv}.

\subsection{The modulo-\texorpdfstring{$p$}{p} derived Hecke operator}
Harris and Venkatesh \cite{hv} define
a derived Hecke operator $T_{(\Z/p\Z)^\times, N}$ on the space
of cusp forms of weight $1$ and level $N$ coprime to $p$
by adding $\cup \mfs_p$ to the definition of the usual Hecke operator
via the pullback and pushforward of two projection maps
$\pi, \pi': X_{\Gamma_1(N) \cap \Gamma_0(p)} \lra X_{\Gamma_1(N)}$:

\adjustbox{scale=1, center}{
  \begin{tikzcd}[row sep = large]
    \begin{array}{c}
      \displaystyle H^0\paren{X_1(N)_{\Z/(p-1)\Z}, \omega(-\cusp)} \\
      \otimes \\
      (\Z/p\Z)^\times
    \end{array}
      \arrow[d, "\pi^*"] \arrow[r, dashed, "T_{(\Z/p\Z)^\times, N}"] &
    \begin{array}{c}
      \displaystyle H^1\paren{X_1(N)_{\Z/(p-1)\Z}, \omega(-\cusp)} \\
      \otimes \\
      (\Z/p\Z)^\times
    \end{array} \\
    \begin{array}{c}
      \displaystyle H^0\paren{X_{0, 1}(p, N)_{\Z/(p-1)\Z}, \omega(-\cusp)} \\
      \otimes \\
      (\Z/p\Z)^\times
    \end{array}
      \arrow[r, "\cup \mfs_{(\Z/p\Z)^\times}"] &
    \begin{array}{c}
      \displaystyle H^1\paren{X_{0, 1}(p, N)_{\Z/(p-1)\Z}, \omega(-\cusp)} \\
      \otimes \\
      (\Z/p\Z)^\times
    \end{array}
      \arrow[u, "\pi'_{*}"]
  \end{tikzcd}
}

\noindent
The Harris--Venkatesh norm concretely encapsulates
the action of the derived Hecke operator:
\[
  \norm{f}^2_{(\Z/p\Z)^\times} = \abrac{f^*, T_{(\Z/p\Z)^\times, N}(f)}_\SD.
\]
This leads to the derived Hecke operator
formulation of Conjecture \ref{conj:hv}.


\section{The \texorpdfstring{$p$}{p}-adic system of weight one forms}
\label{sec:p-adic-forms}

In the modulo-$p$ setting, we applied the Shimura class
$\mfs_{(\Z/p\Z)^\times}$ to the weight-$2$ cusp form
$f(z)f^*(pz)$.
In this section, we construct the $p$-adic analogue
by looking at spaces generated by $f^*(p^n z)$.
Note that this is the only section of this note
in which $p$ is allowed to be $2$ or $3$.

\subsection{Non-vanishing and uniqueness}
Let $f = \sum_n a_n q^n$ be a newform for $\Gamma_1(N)$
of weight $1$ with central character $\omega$
and dual form $f^* = \sum_n \overline{a_n} q^n$.
Let $p$ be a prime that does not divide $N$.
For the field $\Q(f) = \Q(\chi_{\rho_f})$ of coefficients
of $f$ with ring of integers $\Z[f]$,
let $v$ be a place of $\Q(f)$ over $p$
and let $\calo_v$ be the completion of $\Z[f]$ at $v$.
Consider the intersection
$\Gamma_{0, 1}(p^n, N) := \Gamma_0(p^n) \cap \Gamma_1(N)$
and its modular curve $X_{0, 1}(p^n, N)$.

For each integer $n \geq 0$, define the modular form
$f^*_n(z) := f^*(p^n z) = \restr{f^*(z)}{1} \begin{psmallmatrix}p^n & \\ & 1\end{psmallmatrix}$.
The subspace of cusp forms of weight $1$ and level $p^n$
generated by $f^*$ over $\calo_v$ is given by
\[
  V_{f^*, n} := \sum_{i = 0}^n \calo_v f^*_i \subset H^0\paren{X_{0, 1}\paren{p^n, N}_{\calo_v}, \omega(-\cusp)},
\]
where $\cusp$ is the cuspidal divisor.
The spaces $\set{V_{f^*, n}}_{n \geq 0}$ form a projective system
under trace maps;
take the projective limit,
\[
  \wh{V}_{f^*} := \varprojlim_n V_{f^*, n}.
\]
The polynomial
$X^2 - a_p X +\omega(p) p$
(which should not be mistaken for the Hecke polynomial for weight $1$)
has a unique root $\alpha$ in $\calo_v^\times$ by Hensel's lemma;
it appears in calculations of the trace map between levels of
our projective system.

\begin{thm}
  \label{thm:V}
  The space $\wh{V}_{f^*}$ is non-zero
  only if $v$ is ordinary.
  In this case, it is one-dimensional and generated by an element $\wh{f^*} \in \wh{V}_{f^*}$
  whose image in $V_{f^*, n}$ is given by
  \[
    \wh{f^*}_n:= \alpha^{1 - n} f^*_n - \alpha^{-n} \omega(p) f^*_{n - 1}.
  \] 
\end{thm}

\begin{rem}
  In the proof of Theorem \ref{thm:V}, we will only use the fact
  that $f^*$ has eigenvalue $a_p$ under $\rmt_p$.
  Thus, Theorem \ref{thm:V} actually holds for any old form
  $\varphi$ generated by $f^*$ such that the level of $\varphi$
  is coprime to $p$. Also, the condition on ordinariness
  (i.e. $\ord_v(a_p) = 0$)
  is only relevant for $p = 2$ or $3$ since a
  newform of weight $1$ is ordinary at the other primes.
\end{rem}

The key to proving Theorem \ref{thm:V} is the calculation of the trace map
from level $n+1$ to level $n$.

\begin{lem}
  \label{lem:trace-fn+1}
  Consider the trace map
  \[
    \Tr_{n + 1}: V_{f^*, n + 1} \lra V_{f^*, n}, \qquad n \geq 1.
  \]
  \begin{enumerate}[(a)]
    \item Then for all $i \leq n$,
      \[
        \Tr_{n + 1}(f^*_i) = p f_i^*;
      \]
    \item and
      \[
        \Tr_{n + 1}(f^*_{n + 1}) = a_p f_n^* - \omega(p) f_{n - 1}^*.
      \]
  \end{enumerate}
\end{lem}

\begin{proof}
  (a): This is due to the fact that $[\Gamma_{0, 1}(p^n, N) : \Gamma_{0, 1}(p^{n + 1}, N)] = p$.

  (b): Since $\Gamma_{0, 1}(p^{n + 1}, N) \setminus \Gamma_{0, 1}(p^n, N)$
  is represented by elements
  $\gamma_i := \begin{psmallmatrix} 1 & 0 \\ p^n iN & 1 \end{psmallmatrix}$
  for $0 \leq i \leq p - 1$,
  \[
    \Tr_{n + 1}(f^*_{n + 1})
      := \sum_{\gamma \in \Gamma_{0, 1}(p^{n + 1}, N) \setminus \Gamma_{0, 1}(p^n, N)}
        \restr{f^*_{n + 1}}{1} \gamma
      = \sum_i \restr{f^*}{1} \paren{\begin{pmatrix} p^{n + 1} & \\ & 1 \end{pmatrix} \gamma_i}.
  \]
  Notice that
  \[
    \begin{pmatrix} p^{n + 1} & \\ & 1 \end{pmatrix} \gamma_i
      = \begin{pmatrix} p^{n + 1} & \\ p^n iN & 1 \end{pmatrix}
      = \begin{pmatrix} p & \\ iN & 1 \end{pmatrix}
        \begin{pmatrix} p^n & \\ & 1 \end{pmatrix},
  \]
  and the Hecke operator $\rmt_p$ is presented by the
  left $\Gamma_1(N)$-cosets represented by
  $\begin{pmatrix} p & \\ iN & 1 \end{pmatrix}$
  and $\begin{pmatrix} 1 & \\ & p \end{pmatrix}$. Then
  \begin{align*}
    \Tr_{n + 1}(f^*_{n + 1})
      &= \sum_i \restr{f^*}{1} \paren{\begin{pmatrix} p^{n + 1} & \\ & 1 \end{pmatrix} \gamma_i} \\
      &= \sum_i \restr{f^*}{1} \paren{\begin{pmatrix} p & \\ iN & 1 \end{pmatrix}
        \begin{pmatrix} p^n & \\ & 1 \end{pmatrix}} \\
      &= \restr{\paren{\rmt_p f^* - \restr{f^*}{1} \begin{pmatrix} 1 & \\ & p \end{pmatrix}}}{1}
        \begin{pmatrix} p^n & \\ & 1 \end{pmatrix} \\
      &= \rmt_p \restr{f^*}{1} \begin{pmatrix}p^n & \\ & 1 \end{pmatrix}
        - \restr{f^*}{1} \begin{pmatrix} p^n & \\ & p \end{pmatrix} \\
      &= a_p f^*_n - \omega(p) f^*_{n - 1}.
  \end{align*}
\end{proof}

\begin{proof}[Proof of Theorem \ref{thm:V}]
  For any $m \geq n$, let $\Tr_{m, n}: V_{f^*, m} \lra V_{f^*, n}$ denote the
  $(m-n)$-fold composition of trace maps
  \[
    \begin{tikzcd}[row sep=tiny]
      \Tr_{m, n}: \quad V_{f^*, m} \arrow[r, "\Tr_m"] &
      V_{f^*, m-1} \arrow[r, "\Tr_{m-1}"] &
      \ldots  \arrow[r, "\Tr_{n+1}"] &
      V_{f^*, n}.
    \end{tikzcd}
  \]
  Let $\wh{V}_{f^*, n}$ denote the intersection of the images
  $\Im(\Tr_{m, n})$ over all $m \geq n$. 
  Then the projective limit $\wh{V}_{f^*} \neq 0$ if and only if
  $\wh{V}_{{f^*}, n} \neq 0$ for some $n$.

  Let $\pi$ be a uniformizer of $\calo_v$. 
  Consider the case that $\ord_v (a_p) > 1$.
  Then by Lemma \ref{lem:trace-fn+1},
  $\Im(\Tr_{n + 2, n}) \subset \pi V_{{f^*}, n}$ for all $n \geq 1$.
  This shows that $\Im(\Tr_{n + 2k, n}) \subset \pi^k V_{{f^*}, n}$,
  so $\wh{V}_{{f^*}, n} := \bigcap_{m \geq n} \Im(\Tr_{m, n}) = 0$
  and $\wh{V}_{f^*} = 0$ too.

  Now consider the case $\ord_v(a_p) = 0$.
  Take the basis for $V_{{f^*}, n + 1}$ given by,
  \begin{align*}
    (\wh{{f^*}})_0 &:= {f_0^*}, \\
    (\wh{{f^*}})_i &:= \alpha^{1 - i} {f_i^*}
      - \alpha^{-i} \omega(p) {f_{i - 1}^*} & \text{for } 0 < i \leq n + 1.
  \end{align*}
  By Lemma \ref{lem:trace-fn+1}, $\Tr_{n + 1}({f_i^*}) = p {f_i^*}$ for $i < n + 1$
  and $\Tr_{n + 1}({f}_{n + 1}^*) = a_p {f}_n^* - \omega(p) {f}_{n - 1}^*$, so
  \begin{align*}
    \Tr_{n + 1}\paren{(\wh{{f^*}})_{n + 1}}
      &= \alpha^{1 - n - 1} \Tr_{n + 1} ({f_{n + 1}^*})
        - \alpha^{-n - 1} \omega(p) \Tr_{n + 1} ({f_n^*}) \\
      &= \alpha^{-n} \paren{a_p {f^*}_n-\omega(p){f_{n - 1}^*}}
        - \alpha^{-n - 1} \omega(p) p {f^*}_n \\
      &= \alpha^{-n - 1} \paren{\alpha a_p-\omega(p)p} {f_n^*}
        - \alpha^{-n} \omega(p){f^*}_{n - 1} \\
      &= \alpha^{-n - 1} \alpha^2 {f_n^*}-\alpha^{-n} \omega(p){f_{n - 1}^*} \\
      &= (\wh{{f^*}})_n.
  \end{align*}
  Then the sequence $\{(\wh{{f^*}})_n\}_{n \geq 0}$ defines
  an element $\wh{{f^*}} \in \wh{V}_{f^*}$ and
  each $\wh{V}_{{f^*}, n}$ is equal to $\calo_v \cdot (\wh{{f^*}})_n$.
\end{proof}

\subsection{An extended remark on ordinary primes}
The ordinariness of $a_p$ in Theorem \ref{thm:V} can be characterized as follows.
Recall that $f$ corresponds to a two-dimensional complex representation $\rho_f$
by Deligne--Serre with $a_p = \Tr(\rho_f(\Frob_p))$. In fact,
the $p$-th Fourier coefficient $a_p$ of $f$
is the sum of two roots of unity whose order divides $\# \Im(\rho_f)$:
\[
  a_p = \zeta_1 + \zeta_2 = \zeta_1 \cdot (1 - \xi),
\]
where $\xi = -\frac{\zeta_2}{\zeta_1}$.
It is known that nonzero $a_p$ is not invertible for some
$p$-adic norm on $\Q(f)$
if and only if the order of $\xi$ is a power of $p$.
This shows that $p$ is ordinary if $a_p \neq 0$ and $p \nmid 2 \#\Im(\rho_f)$.

When $\rho_f$ has dihedral image, more can be said about
when $a_p$ is non-zero based on the splitting of $p$ in $K$.
In the dihedral case, $\rho_f \cong \Ind_{G_K}^{G_\Q}(\chi)$
for a finite character $\chi$ of $\Gal(K^\ab/K)$
and a quadratic number field $K/\Q$.
Here, $L(f, s) = L(\chi, s)$ and 
\[
  f = \sum_{\mfa \in S_\chi} \chi(\mfa) q^{\rmn(\mfa)},
\]
where $\mfa$ runs through the set $S_\chi$
of ideals of $\calo_K$ coprime to the conductor of $\chi$.
Therefore,
\[
  a_n = \sum_{\mfa: \rmn(\mfa) = n} \chi(\mfa),
\]
so $a_p = 0$ when $p$ is inert in $K$
and $a_p = \chi(\mfp)+\chi(\bar\mfp)$
when $p \calo_K= \mfp + \overline{\mfp}$.


\section{The \texorpdfstring{$p$}{p}-adic Shimura class and derived Hecke operator}
\label{sec:p-adic-Shimura}

In this section, we define the $p$-adic Shimura classes
which play a central role in defining $p$-adic Hecke operators.
We then define the $p$-adic norm
and $p$-adic derived Hecke operators.
This completes the formulation of
Conjecture \ref{conj:p-adic-hv}.
$\Z_p$ is used throughout this section for notational convenience,
but one may also take $\Z_p \otimes \Z[f]$
(analogously to $(\Z/p\Z)^\times \otimes \Z[f])$.

\subsection{The \texorpdfstring{$\Z_p^\times$}{Zp-times} Shimura class}
Let $p$ be a prime.
Consider the tower of coverings of modular curves for $n \in \N$,
\[
  \pi_n: X_1 \paren{p^n} \lra X_0\paren{p^n}.
\]
These covers are not \'{e}tale when $n > 1$
even after inverting $6N$,
but are finite flat (cf. \cite[Proposition 11.6]{mazur} in the $n=1$ case)
and define elements of flat cohomology,
\[
  \mfs_{\paren{\Z/p^n\Z}^\times, \, \fl }
    \in H^1_\fl \paren{X_0\paren{p^n}, \paren{\paren{\Z/p^n\Z}^\times}_a}.
\]
Here, we denoted the image of $\paren{\Z/p^n\Z}^\times$
in $\G_a$ by the quotient sheaf
$\paren{\paren{\Z/p^n\Z}^\times}_a$.
In the \'{e}tale case,
an unramified cover of connected $X = X_1(p)$ over $Y = X_0(p)$
with Galois group $(\Z/p\Z)^\times$
gives a morphism from the pointed fundamental group
$\pi_0(X, x_0)$ to $(\Z/p\Z)^\times$, which gives
an element of $H^1(X, (\Z/p\Z)^\times) \cong
\Hom(\pi_1(X, x_0), (\Z/p\Z)^\times)$ since the Galois group is
abelian. Now with $X = X_1(p^n)$ over $Y = X_0(p^n)$ with structure
group $(\Z/p^n\Z)^\times$, the cover is not
\'{e}tale but it is finite flat. The fiber product
$X \times_Y X$ is a union of subvarieties isomorphic to
$X$ with a group action by some quotient of the
original group $(\Z/p^n\Z)^\times$.
Fix one copy of $X$ in $X \times_Y X$; the other copies of $X$
come from actions of the quotient group,
and $X \times_Y X$ is a group scheme over $X$.
This gives a cochain from $X \times_Y X$
to the quotient group and an element of
$H^1_\fl \paren{X_0\paren{p^n}, \paren{\paren{\Z/p^n\Z}^\times}_a}$.

The compatibility of $\set{\pi_n}_{n \geq 0}$
gives an element $\mfs_{\Z_p^\times, \, \fl}$
in the (partially) complete cohomology
\[
  \wh{H}^1_\fl \paren{X_0\paren{p^\infty}, \Z_p^\times}
    := \varinjlim_n \varprojlim_m H^1_\fl \paren{X_0\paren{p^n}, \paren{\paren{\Z/p^m\Z}^\times}_a}.
\]
One can then view the Shimura class
as an element of Zariski cohomology,
similar to what
was done for the $n=1$ \'{e}tale case
of Harris--Venkatesh \cite[Section 3.1]{hv}.

\subsection{The \texorpdfstring{$\Z_p$}{Zp} Shimura class}
We can obtain a $\Z_p$ Shimura class from this
$\Z_p^\times$ Shimura class.
Consider the natural decomposition
\[
  \Delta \times \Z_p \iso \Z_p^\times,
\]
where $\Delta$ is the torsion subgroup of $\Z_p^\times$.
In particular,
$\Delta$ is isomorphic to $\mu_{p-1} \cong (\Z/p\Z)^\times$ when $p$ is odd,
and to $\mu_2$ when $p=2$.
In the other component,
the map $\Z_p \lra \Z_p^\times$ is given by $x \mapsto (1 + p)^x$ when $p$ is odd
and $x \mapsto 5^x$ when $p = 2$.
Then we have a product
\[
  \mfs_{\Z_p^\times, \, \fl} = \mfs_{\Delta, \, \fl} \times \mfs_{\Z_p, \, \fl}.
\]

The element $\mfs_{\Delta, \, \fl} = \mfs_{\Delta, \, \et}
\in H^1_\et (X_0(p), (\Z/p\Z)^\times)$
gives the mod-$p$ Shimura class
previously studied by \cite{mazur,merel2,hv,marcil,dhrv,horawa,lecouturier-hv,zhang-hv,zhang-hvs,zhang-rs} among others.
Here, we focus on the $\Z_p$ Shimura class $\mfs_{\Z_p, \, \fl}$ lying in
the cohomology group
\[
  \wh{H}^1_\fl \paren{X_0\paren{p^\infty}, \Z_p}
    := \varinjlim_n \varprojlim_m H^1_\fl \paren{X_0\paren{p^n}, \Z/p^m\Z}.
\]
This group has a natural map to the completed cohomology of coherent sheaves:
\[
  \varinjlim_n \varprojlim_m H^1 \paren{X_0\paren{p^n}_{\Z/p^m\Z}, \calo}.
\]
Using Serre duality, 
\[
  H^1 \paren{X_0\paren{p^n}_{\Z/p^m\Z}, \calo}
    \cong \Hom\paren{H^0\paren{X_0\paren{p^n}, \Omega}, \Z/p^m\Z}.
\]
Thus for $\wh{H}^0 (X_0(p^\infty), \Omega) := \varinjlim_n H^0 (X_0(p^n), \Omega)$,
the Shimura class $\mfs_{\Z_p}$ can also be viewed as an element in its dual space
\begin{align*}
  \Hom_{\Z_p} \paren{\wh{H}^0 \paren{X_0\paren{p^\infty}, \Omega}, \Z_p},
\end{align*}
where trace maps define the projective system.

One can perform all of these Shimura class constructions with
additional level structure.
It will be useful in the $p$-adic consideration of the
Harris--Venkatesh conjecture in the following section
to use these constructions
for $\Gamma_{0, 1}(p^n, N) := \Gamma_0(p^n) \cap \Gamma_1(N)$
and its modular curve $X_{0, 1}(p^n, N)$
instead of $\Gamma_0(p^n)$ and $X_0(p^n)$.


\subsection{The \texorpdfstring{$p$}{p}-adic norm}

We apply Theorem \ref{thm:V} to the dual form $f^* = \sum_n \overline{a}_n q^n$ of $f$.
Assuming that $\ord_v(\overline{a}_p) = 0$, we have the element
$\wh{f}^* \in \wh{H}^0(X_0(p^\infty), \omega(-\cusp))$ and in particular
\[
  f \cdot \wh{f}^* \in \wh{H}^0\big(X_0(p^\infty), \Omega \big).
\]
Recall that $\calo_v$ is the completion of $\Z[f]$ at $v$;
using the $\Z_p$-Shimura class $\mfs_{\Z_p, \fl}
\in \Hom_{\Z_p}(\wh{H}^0 (X_0(p^\infty), \Omega), \Z_p)$
from Section \ref{sec:p-adic-Shimura},
we can define a $p$-adic period of $f$,
\[
  \norm{f}_{\Z_p}^2 := \mfs_{\Z_p} \big(f \cdot \wh{f}^*\big) \in \calo_v.
\]
Concretely, what we have done in Section \ref{sec:p-adic-Shimura} is
the construction of the class
$\mfs_{\Z/p^n\Z, \, \fl} \in H^1(X_{0, 1}(p^{n + 1}, N), \Z/p^n\Z)$
using the \'{e}tale cover $X_1(p^n) \lra X_0(p^n)$,
and then the demonstration that this class induces an element
$\mfs_{\Z/p^n\Z} \in \Hom(H^0(X_{0, 1}(p^{n + 1}, \Omega), \Z/p^n\Z)$
through which we obtain
\[
  \norm{f}_{\Z_p}^2
    = \varprojlim_n \mfs_{\Z/p^n\Z}
      \Big(\big(\alpha^{1-n} f^*_n(z) - \alpha^{-n} \omega(p) f^*_{n-1}(z)\big)
      \, f(z)dz\Big).
\]

With these constructions in place,
we are now ready to formulate
Conjecture \ref{conj:p-adic-hv} in terms of the $p$-adic norm:
we ask whether
\[
  \norm{f}_{\Z_p}^2 = \Reg_{\Z_p}(u)
\]
for some element $u \in \calu(\Ad(\rho_f)) \otimes \Q$
and all sufficiently large primes $p$.
In other words,
the $\Z_p$ Shimura class should act on
weight $1$ forms compatibly with
the $(\Z/p\Z)^\times$ Shimura class
in the Harris--Venkatesh conjecture,
and there is a combined equality
for all sufficiently large primes $p$:
\[
  \norm{f}_{\Z_p^*}^2 = \Reg_{\Z_p^*}(u).
\]

\subsection{The \texorpdfstring{$p$}{p}-adic derived Hecke operator}
Define the subspace of $p$-adic ordinary elements
in $H^i(X_1(N), \omega(-\cusp)) \otimes \Z_p$ to be
\[
  H^i(X_1(N), \omega(-\cusp))^\ord
    := \bigcap_{n \geq 0} \rmt_p^n \, H^i \big(X_1(N), \omega(-\cusp)\big) \otimes \Z_p.
\]
In this space, we have the action of the invertible operator $\rmt_p$ and $\pair{p}$.
There is a unique invertible operator $A$ on $H^i(X_1(N), \omega(-\cusp))^\ord$
satisfying the equation:
\[
  A^2 - \rmt_p A + \pair{p}p = 0.
\]
Such an $A$ can be either defined by the binomial formula,
\[
  A = \frac{\rmt_p}{2} \paren{1 + \sqrt{1 - \frac{4\pair{p}p}{\rmt_p^2}}}
    = \rmt_p-\sum_{n \geq 1} \frac {1 \cdot 3 \cdots (2n-3)}{2^{1-n} n!}
      \paren{\frac{\pair{p} p}{\rmt_p^2}}^n,
\]
or by the limit  $A = \lim_{n \rightarrow \infty}A_n$ with $A_n$ the invertible operator defined by 
\begin{align*}
  A_0 &= \rmt_p, \\
  A_n &= \rmt_p - \frac{\pair{p}p}{A_{n - 1}}, & n > 0.
\end{align*}

Serre duality gives a perfect pairing:
\[
  \pair{\, \cdot\, , \, \cdot \,}_\SD:
    H^0\big(X_1(N), \omega(-\cusp)\big)^\ord \otimes
    H^1\big(X_1(N), \omega(-\cusp)\big)^\ord
    \lra \Z_p.
\]
Using the operator $A$ and $f_n(z) := f(p^nz)$, we have a lifting
\[
  \begin{tikzcd}[row sep=tiny]
    H^0\big(X_1(N), \omega(-\cusp)\big)^\ord \arrow[r]
      & \varprojlim_n H^0\big(X_{0, 1}(p^n, N), \omega(-\cusp)\big) \\
    f \arrow[r, mapsto]
      & \wh{f}:= \varprojlim_n \paren{A^{1 - n} f_n - A^{-n}\pair{p}p f_{n - 1}}.
  \end{tikzcd}
\]
Then we have a derived Hecke operator
\[
  T_{\Z_p, N}: H^0\big(X_1(N), \omega(-\cusp)\big)^\ord
    \lra H^1\big(X_1(N), \omega(-\cusp)\big)^\ord
\]
defined by this lifting and $\cup \mfs_{\Z_p}$
such that for any $f_1, f_2\in H^0(X_1(N), \omega(-\cusp))^\ord$, 
\[
  \mfs_{\Z_p} \big(f_1 \wh{f}_2\big) = \pair{f_2, T_{\Z_p, N}(f_1)}_\SD.
\]
We need only consider new forms $f_1 = f, f_2 = f^*$;
in this case, the prediction that
\[
  \norm{f}_{\Z_p}^2 = \Reg_{\Z_p}(u),
\]
together with
the pairing identity above, yields
the derived Hecke operator formulation of
Conjecture \ref{conj:p-adic-hv} in terms of
$\pair{f^*, T_{\Z_p, N}(f)}_\SD = \Reg(u)$.


\bibliography{bibliography}{}
\bibliographystyle{alpha}

\end{document}